\documentclass[a4paper, 11pt]{amsart}
\usepackage[utf8]{inputenc}
\usepackage[T1]{fontenc}
\usepackage{lmodern}
\usepackage[english]{babel}
\usepackage[dvipsnames]{xcolor}
\usepackage{adjustbox}
\usepackage{letltxmacro}
\usepackage{todonotes}
\LetLtxMacro\todonotestodo\todo
\renewcommand{\todo}[2][]{\todonotestodo[#1]{TODO: {#2}}}
\usepackage{enumerate}
\usepackage{hyperref}
\usepackage{url}
\usepackage{mathrsfs}
\usepackage{amssymb}

\PassOptionsToPackage{hyphens}{url}\usepackage{hyperref}

\theoremstyle{definition}
\newtheorem{theorem}{Theorem}

\makeatletter
\newtheorem*{rep@theorem}{\rep@title}
\newcommand{\newreptheorem}[2]{%
\newenvironment{rep#1}[1]{%
 \def\rep@title{#2 \ref{##1}}%
 \begin{rep@theorem}}%
 {\end{rep@theorem}}}
\makeatother

\newreptheorem{theorem}{Theorem}

\newtheorem{lemma}{Lemma}[section]
\newtheorem{proposition}[lemma]{Proposition}
\newtheorem{corollary}[lemma]{Corollary}
\newtheorem{fact}[lemma]{Fact}

\newtheorem{remark}[lemma]{Remark}

\newtheorem*{claim*}{Claim}
\newtheorem{notation}[lemma]{Notation}
\newtheorem*{theorem*}{Theorem}
\newtheorem*{corollary*}{Corollary}
\newtheorem*{lemma*}{Lemma}

\newtheorem*{problem*}{Problem}

\newtheorem{problem}[lemma]{Problem}

\newcommand{\Q}{\mathbb{Q}}
\newcommand{\Z}{\mathbb{Z}}
\newcommand{\N}{\mathbb{N}}

\DeclareMathOperator{\Int}{Int}

\title[IVP\MakeLowercase{s} on VR\MakeLowercase{s} of global fields with prescribed lengths of factorizations]{Integer-valued polynomials on valuation rings of global fields with prescribed lengths of factorizations}
\thanks{\textit{Mathematics subject classification.} primary 11R09, 11C08, 13A05; secondary 12E05, 13F20, 13F05}
\thanks{\textit{Key words.} integer-valued polynomials, global fields, irreducible polynomials, factorizations, discrete valuations domains}
\author{Victor Fadinger}
\thanks{V. Fadinger is supported by the Austrian Science Fund (FWF): W1230}

\author{Sophie Frisch}
\thanks{S. Frisch is supported by the Austrian Science Fund (FWF): P~35788}

\author{Daniel Windisch}
\keywords{}
\thanks{D.~Windisch is supported by the Austrian Science Fund (FWF): P~30934}

\begin{document}

\begin{abstract}
Let $V$ be a valuation ring of a global field $K$. We show that for all positive integers $k$ and $1 < n_1 \leq \ldots \leq n_k$ there exists an integer-valued polynomial on $V$, that is, an element of $\Int(V) = \{ f \in K[X] \mid f(V) \subseteq V \}$, which has precisely $k$ essentially different factorizations into irreducible elements of $\Int(V)$ whose lengths are exactly $n_1,\ldots,n_k$. In fact, we show more, namely that the same result holds true for every discrete valuation domain $V$ with finite residue field such that the quotient field of $V$ admits a valuation ring independent of $V$ whose maximal ideal is principal or whose residue field is finite. If the quotient field of $V$ is a purely transcendental extension of an arbitrary field, this property is satisfied. This solves an open problem proposed by Cahen, Fontana, Frisch and Glaz in these cases.
\end{abstract}

\maketitle

\section{Introduction}

Non-unique factorization in integral domains has been a recurring topic in commutative ring theory ever since the phenomenon was first discovered in rings of integers in algebraic number fields.
The machinery developed in this setting generalizes to Dedekind domains
and even further, to potentially non-Noetherian, higher-dimensional
analogues of Dedekind domains, namely, Krull domains. Factorizations in Krull domains (more generally in Krull monoids) are well-studied and can be described by combinatorial structure depending only on the divisor class group and the distribution of prime divisors in the classes, see the monograph by Geroldinger and Halter-Koch \cite{GHK}.

In contrast to this, another important
generalization of Dedekind domains, namely, Pr\"ufer domains,
is not amenable to the existing methods and (so far) there is no general theory of non-unique factorization in this case. For this reason, the study of non-unique factorizations in Prüfer domains relies on ad-hoc arguments in each particular case. 
Pr\"ufer domains, characterized by the
condition that every finitely generated ideal is invertible, are
however, an important class of rings: a natural common generalization of
Dedekind domains and valuation rings.

The non-Noetherian Prüfer domain where  non-unique factorization was first studied is the ring of integer-valued polynomials on $\Z$, that is, $\Int(\Z)=\{f\in \Q[X]\mid f(\Z)\subseteq \Z\}$. Building on results by Cahen and Chabert \cite{elasticitycahen} and Chapman and McClain \cite{elasticitychapman}, the second author \cite{integers} showed that every finite multiset of integers $>1$ occurs as the set of lengths (of factorizations into irreducibles) of some polynomial in $\Int(\Z)$. This result was generalized to rings of integer-valued polynomials on Dedekind domains with infinitely many maximal ideals of finite index by Nakato, Rissner and the second author \cite{globalcase}.

The analogous question for integer-valued polynomials on discrete valuation domains with finite residue field is an open problem:
\begin{problem*}\cite[Problem 39]{open}
Analyze and describe non-unique factorization in $\Int(V)$, where $V$ is a DVR with finite residue field. 
\end{problem*}

Note that if $V$ has an infinite residue field then $\Int(V)=V[X]$, which has unique factorizations. In the non-trivial case of finite residue fields not much is known,
except for isolated facts about absolutely and non-absolutely irreducible
elements. An irreducible element is non-absolutely irreducible if some
of its powers factor non-uniquely. In rings of integers in number
fields, non-unique factorization into irreducibles is equivalent to
the existence of non-absolutely irreducible elements as Chapman and Krause~\cite{ChapmanKrause} showed. 
Nakato, Rissner and the second author \cite{graph,splitDVR} characterized when certain irreducible elements are absolutely irreducible in $\Int(V)$.  The binomial polynomials in $\Int(\Z)$ were shown to be absolutely irreducible by Rissner and the third author~\cite{binomial}. 

Returning to sets of lengths in $\Int(D)$, the methods used so far \cite{integers, globalcase} rely heavily on the existence of prime ideals of arbitrarily large index and, hence, do not apply to the case of discrete valuation domains. 

Using combinatorial linear algebra, we are able to approach this problem for discrete valuation domains in certain fields, obtaining the following

\begin{theorem*}
Let $V$ be a discrete valuation domain with finite residue field. Suppose that the quotient field $K$ of $V$ admits a valuation ring independent from $V$ whose maximal ideal is principal. Let $k$ be a positive integer and $1 < n_1 \leq \ldots \leq n_k$ integers. 

Then there exists an integer-valued polynomial $H \in \Int(V)$ which has precisely $k$ essentially different factorizations into irreducible elements of $\Int(V)$ whose lengths are exactly $n_1,\ldots,n_k$.
\end{theorem*}

It is implicit in our proof that the monic irreducible polynomials of degree $n$ lie dense (with respect to the $V$-adic topology on $K$) in the set of all monic polynomials of degree $n$ over a field $K$ as above. This is true in particular for global fields. From the above theorem, we immediately obtain the following corollary.

\begin{corollary*}
The conclusion of the
theorem holds in each of the following cases:
\begin{itemize}
\item[(1)] $V$ is a valuation ring of a global field.
\item[(2)] $V$ is a discrete valuation domain with finite residue field such that the quotient field of $V$ is a purely transcendental extension of an arbitrary field.
\item[(3)] $V$ is a discrete valuation domain with finite residue field such that the quotient field $K$ of $V$ is a finite extension of a field $L$ that admits a valuation ring independent from $V \cap L$ whose maximal ideal is principal or whose residue field is finite.
\end{itemize}

That is, in each of these three cases, for all positive integers $k$ and $1 < n_1 \leq \ldots \leq n_k$, there exists an integer-valued polynomial $H \in \Int(V)$ which has precisely $k$ essentially different factorizations into irreducible elements of $\Int(V)$ whose lengths are exactly $n_1,\ldots,n_k$.

\end{corollary*}

\section{Preliminaries}

\textbf{Factorizations.} We give an informal presentation of factorizations. The interested reader is refered to the monograph by Geroldinger and Halter-Koch \cite{GHK} for a systematic introduction.

Let $R$ be an integral domain and $r\in R$. We say that $r$ is \textit{irreducible} (in $R$) if it cannot be written as the product of two nonunits of $R$. A \textit{factorization} of $r$ is a decomposition
\[r=a_1\cdots a_n\]
into irreducible elements $a_i$ of $R$. In this case $n$ is called the \textit{length} of this factorization of $r$. Let $s$ be a further element of $R$. We say that $r$ and $s$ are associated if there exists a unit $\varepsilon\in R$ such that $r=\varepsilon s$. We want to consider factorizations up to order and associates. In other words two factorizations
\[r=a_1\cdots a_n=u_1\cdots u_m\]
of $r$ are \textit{essentially the same} if $n=m$ and, after re-indexing if necessary, $u_i$ is associated to $a_i$ for all $i\in \{1,\ldots ,n\}$. Otherwise, the factorizations are called \textit{essentially different}.\\

\textbf{Valuations.} Let $K$ be a field. A valuation $\mathsf v$ on $K$ is a map
\[\mathsf v:K^\times\to G\]
where $(G,+,\leq)$ is a totally ordered Abelian group, subject to the following conditions for all $a,b\in K^\times$:
\begin{enumerate}
\item $\mathsf v(a\cdot b)=\mathsf v(a) +\mathsf v(b)$ and
\item $\mathsf v(a+b)\geq \inf\{\mathsf v(a),\mathsf v(b)\}$.
\end{enumerate}
The set $\{0\}\cup\{x\in K^\times\mid \mathsf v(x)\geq 0\}$ is called the \textit{valuation ring} of $\mathsf v$. It is a subring of $K$ with quotient field $K$ and unique maximal ideal $\{0\}\cup\{x\in K^\times\mid \mathsf v(x)> 0\}$. 

The group $\mathsf{v}(K^\times)$ is called the \textit{value group} of $\mathsf{v}$.

We will often use implicitly the following fact about valuations, which follows from the definition by an easy exercise: If $\mathsf v$ is a valuation on $K$ and $a,b\in K$ are such that $\mathsf v(a)\neq \mathsf v(b)$ then 
$$\mathsf v(a+b)= \inf\{\mathsf v(a),\mathsf v(b)\}.$$

If $v(K^\times)\cong \Z$ we call $\mathsf v$ a \textit{discrete valuation} (by the more precise terminology of Bourbaki it would be a \textit{discrete rank one valuation}). If $\mathsf v$ is a discrete valuation on $K$, then there exists a valuation $\mathsf w:K^\times \to \Z$ with $\mathsf w(K^\times)=\Z$ and the same valuation ring as $\mathsf v$. We call $\mathsf w$ the \textit{normalized valuation} of this valuation ring.

For a general introduction to valuations, see \cite{Bourbaki}.\\

\textbf{Discrete valuation domains.} An integral domain $V$ is said to be a \textit{discrete valuation domain} (DVR) if it satisfies one of the following equivalent statements:
\begin{enumerate} 
\item $V$ is the valuation ring of a discrete valuation on a field.
\item $V$ is a unique factorization domain with a unique prime element up to associates.
\item $V$ is a principal ideal domain with a unique non-zero prime ideal.
\item $V$ is a local Dedekind domain but not a field.
\end{enumerate}

If $V$ is a DVR with normalized valuation $\mathsf v$ then the prime elements of $V$ (which are all associated) are precisely the elements $p\in V$ with $\mathsf v(p)=1$.

If $M$ is the unique maximal ideal of $V$ then $V/M$ is called its \textit{residue field}.\\

\textbf{Fields.} By a \textit{global field} we mean a finite extension either of the field of rational numbers $\Q$ or of a field of rational functions $\mathbb F(T)$ in one variable over a finite field $\mathbb F$. The first type is refered to as \textit{algebraic number field} and the second as \textit{algebraic function field}. Note that every valuation ring of a global field is a discrete valuation domain with finite residue field.\\


\textbf{Integer-valued polynomials.} Let $R$ be an integral domain with quotient field $K$. The set
\[\text{Int}(R)=\{f\in K[X]\mid f(R)\subseteq R\}\] 
is a subring of $K[X]$ and called the \textit{ring of integer-valued polynomials} on $R$.
Let $V$ be the valuation ring of valuation $\mathsf v$ on a field $K$. Every element $f\in K[X]$ can be written in the form $f=\frac{g}{d}$, where $g\in V[X]$ and $d\in V\setminus\{0\}$. It is immediate that $f\in \Int(V)$ if and only if $\min_{a\in V} \mathsf v(f(a))\geq \mathsf v(d)$.

For a detailed treatment of integer-valued polynomials we refer to the monograph by Cahen and Chabert \cite{cahen}.

\section{Glueing of polynomials}

Let $V$ be a discrete valuation domain with finite residue field. Let $K$ be the quotient field of $V$. The purpose of this section is to construct monic polynomials in $V[X]$ of a given degree that are irreducible over $K$ and behave similarly as a given product of linear factors with respect to the valuation of $V$. We can solve this problem in two cases, see Lemma~\ref{lemma:glueing} and Lemma~\ref{lemma:glueing-finite}. We understand this as a sort of glueing process of linear factors into something indecomposable.

\begin{remark}\label{remark:valuation}
Let $K$ be a field and $W$ a valuation domain of $K$ with corresponding valuation $\mathsf{w}$ and maximal ideal $M$. The following are easily seen to be equivalent:
\begin{itemize}
\item[(a)] $M$ is principal and $W$ is not a field.
\item[(b)] The value group of $\mathsf{w}$ has a minimal element $>0$.
\item[(c)] $M \neq M^2$.
\end{itemize}
We, therefore, use these three properties interchangeably throughout the manuscript.
\end{remark}

\begin{proof}
Since (a) and (b) are clearly equivalent, we only have to argue that (a) is equivalent to (c). Suppose that $M = M^2$. Either $M = (0)$, in which case $W$ is a field, or $M \neq (0)$. In this case, let $x \in M\setminus \{0\}$, that is, $\mathsf{w}(x) > 0$.  Since $x \in M^2$, there exist $x_1,\ldots,x_n,y_1,\ldots,y_n \in M$ such that $x = \sum_{i = 1}^n x_iy_i$. Now,
\[ \mathsf{w}(x) \geq \min_i \mathsf{w}(x_iy_i) \]
and hence $\mathsf{w}(x) \geq \mathsf{w}(x_i) + \mathsf{w}(y_i)$ for some $i$. Since $\mathsf{w}(x_i), \mathsf{w}(y_i) >0$, it follows that $\mathsf{w}(x)$ cannot be minimal $>0$.

Conversely, suppose that the value group of $\mathsf{w}$ does not have a minimal element $>0$. Let $x \in M$ and $y \in M$ with $\mathsf{w}(x) > \mathsf{w}(y) > 0$. Pick $z \in M$ with $\mathsf{w}(z) = \mathsf{w}(x) - \mathsf{w}(y) > 0$. Then $\mathsf{w}(x) = \mathsf{w}(yz)$ and therefore there exists $\varepsilon \in W^\times$ such that $x = (\varepsilon y) \cdot z \in M^2$.
\end{proof}

The following is a known irreducibility criterion, see the text book by Matsumura~\cite{matsumura}. We include a simple proof for the special case we need.

\begin{lemma}\label{lemma:Eisenstein}
\cite[§29, Lemma 1 and its proof]{matsumura} Let $W$ be an integrally closed local domain with quotient field $K$ and let $N$ be the maximal ideal of $W$. Let $F = \sum_{i=0}^n d_i X^i \in W[X]$ with the following properties:
\begin{itemize}
\item[(i)] $d_n \notin N$.
\item[(ii)] $d_i \in N$ for all $i \in \{0, \ldots, n-1\}$.
\item[(iii)]  $d_0 \notin N^2$.
\end{itemize}
Then $F$ is irreducible in $W[X]$ and in $K[X]$.
\end{lemma}

\begin{proof}
We first show that $F$ is not a product of two non-constants in $W[X]$. Assume to the contrary that $F = ST$ where $S,T \in W[X]\setminus W$. Then
\[ \overline{S} \cdot \overline{T} = \overline{F} = \overline{d_n} X^n, \]
where $\overline{ \ \cdot \ }$ denotes the reduction modulo $N$. Since $W/N$ is a field, it follows that
\[\overline{S} = \overline{b}X^s, \overline{T} = \overline{c}X^t, \]
where $\overline{b} \cdot \overline{c} = \overline{d_n} \neq 0$ and $s+t = n$, $s \neq n \neq t$. So the constant terms of $S$ and $T$ lie in $N$, contradicting $d_0 \notin N^2$. Since $d_n$ is not in $N$, we also cannot factor out a non-unit constant in $W[X]$. This shows that $F$ is irreducible in $W[X]$.

Since $d_n \notin N$, we can assume without loss of generality that $F$ is monic. Now $W$ is integrally closed, whence $F$ is also irreducible in $K[X]$~\cite[Chapter 5, § 1.3, Proposition 11]{Bourbaki}.
\end{proof}

\begin{lemma}\label{lemma:glueing}
Let $V$ be a discrete valuation domain with finite residue field. Suppose that the quotient field $K$ of $V$ admits a valuation domain independent from $V$ whose maximal ideal is non-zero principal. Let $\mathsf v: K^\times \to \Z$ be the normalized valuation of $V$ and $R_1,\ldots,R_q$ the residue classes of $V$. For each $k \in \{1,\ldots,q\}$ choose $r_k \in R_k$ arbitrary. 

Let $f = \prod_{i=1}^n (X-a_i)$, where $a_1,\ldots,a_n \in V$ with $\mathsf{v}(r_k-a_i) \in \{0,1\}$ for all $i,k$.

Then there exists $F \in V[X]$ irreducible over $K$ with $\deg(F) = n$ such that 
\[ \min \{ \mathsf v(f(a)) \mid a \in  R_k\} = \min\{ \mathsf v(F(a)) \mid a\in R_k\}=\mathsf v(F(r_k))\] 
for all $k \in \{1,\ldots,q\}$.
\end{lemma}

\begin{proof}
Let $f = \prod_{i=1}^n (X-a_i) = \sum_{i = 0}^n b_i X^i$. Let $\mathsf{w}$ be a valuation on $K$ independent from $\mathsf{v}$ whose value group admits a minimal element $\pi>0$, see Remark~\ref{remark:valuation}. Choose $c_0,\ldots,c_{n-1} \in K$ such that $\mathsf{v}(c_i) = n+1$ and $\mathsf{w}(b_i+c_i) = \pi$ for all $i \in \{0,\ldots,n-1\}$ which is possible by the Approximation Theorem for independent valuations~\cite[Theorem 22.9]{Gilmer}. Let $F = f + \sum_{i = 0}^{n-1} c_i X^i$ which is irreducible over $K$ by applying Lemma~\ref{lemma:Eisenstein} with respect to $\mathsf{w}$. Clearly, $\deg(F) = n$.

Let $k \in \{1,\ldots,q\}$. Then $\mathsf{v}(f(r_k)) = \min \{ \mathsf v(f(a)) \mid a \in  R_k\}$. Also 
\[ \mathsf{v}(F(r_k)) = \min \{\mathsf{v}(f(r_k)),\mathsf{v}(\sum_{i = 0}^{n-1} c_i r_k^i)\} = \mathsf{v}(f(r_k)),\]
because $\mathsf{v}(c_i) = n+1$ and therefore $\mathsf{v}(f(r_k)) \leq n < \mathsf{v}(\sum_{i = 0}^{n-1} c_i r_k^i)$. If now $a \in R_k$ then
\begin{align*}
\mathsf{v}(F(a)) &= \mathsf{v}(f(a) + \sum_{i = 0}^{n-1} c_i a^i) \\
				&\geq \min \{\mathsf{v}(f(a)),\mathsf{v}(\sum_{i = 0}^{n-1} c_i a^i)\} \\
				&\geq \mathsf{v}(f(r_k)) = \mathsf{v}(F(r_k)).
\end{align*}
\end{proof}

\begin{lemma}\label{lemma:glueing-finite}
Let $V$ be a discrete valuation domain with finite residue field whose quotient field $K$ admits a valuation domain $W$ independent from $V$ whose residue field is also finite. Let $\mathsf v: K^\times \to \Z$ be the normalized valuation of $V$ and $R_1,\ldots,R_q$ the residue classes of $V$. For each $k \in \{1,\ldots,q\}$ choose $r_k \in R_k$ arbitrary.

Let $f = \prod_{i=1}^n (X-a_i)$, where $a_1,\ldots,a_n \in V$ with $\mathsf{v}(r_k-a_i) \in \{0,1\}$ for all $i,k$.

Then there exists $F \in V[X]$ irreducible over $K$ with $\deg(F) = n$ such that 
\[\min \{ \mathsf v(f(a)) \mid a \in  R_k\} = \min\{ \mathsf v(F(a)) \mid a\in R_k\}=\mathsf v(F(r_k))\] 
for all $k \in \{1,\ldots,q\}$.
\end{lemma}

\begin{proof}
Let $f = \prod_{i=1}^n (X-a_i) = \sum_{i = 0}^n b_i X^i$. Let $\mathsf{w}$ be a valuation on $K$, independent of $V$, with finite residue field, valuation ring $W$, $P$ the maximal ideal of $W$ and $R = V \cap W$. We construct a monic polynomial $F = X^n + \sum_{i = 0}^{n-1} F_i X^i \in R[X]$ that is irreducible in $K[X]$ and satisfies $\mathsf{v}(b_i - F_i) > n$ for $i \in \{0,\ldots,n-1\}$ and then, we show that this $F$ has the required properties.

It is well-known that there exist irreducible polynomials of every degree over a finite field. In particular, we can choose $g = X^n + \sum_{i = 0}^{n-1} g_i X^i \in W[X]$ a monic polynomial of degree $n$ that is irreducible in $(W/P)[X]$. By the Approximation Theorem for independent valuations~\cite[Theorem 22.9]{Gilmer}, there exist $F_0,\ldots,F_{n-1} \in K$ such that $\mathsf{v}(b_i - F_i) > n$ and $\mathsf{w}(g_i - F_i) > 0$. Let $F = X^n + \sum_{i = 0}^{n-1} F_i X^i$. Then $F \in R[X]$ and $F$ is irreducible in $(W/P)[X]$, because its reduction $\overline{F}$ modulo $P$ is the same as that of $g$. 

We first show that $F$ is irreducible in $W[X]$. Let $F = ST$ where $S,T \in W[X]$. Then $\overline{S} \cdot \overline{T} = \overline{F}$. Since $\overline{F}$ is irreducible, it follows that either $\overline{S}$ or $\overline{T}$ is a unit in $(W/P)[X]$. Since $F$ is monic, it follows that either $S$ or $T$ is in fact a unit in $W[X]$, whence $F$ is irreducible in $W[X]$. $W$ is integrally closed and, therefore, $F$ is also irreducible in $K[X]$~\cite[Chapter 5, § 1.3, Proposition 11]{Bourbaki}.

Now, for each $k \in \{1,\ldots,q\}$,
\[ \mathsf{v}(f(r_k) - F(r_k)) = \mathsf{v}(\sum_{i = 0}^{n-1} (b_i - F_i) r_k^i) > n \geq \mathsf{v}(f(r_k)) \geq \min\{\mathsf{v}(f(r_k)),\mathsf{v}(F(r_k))\}. \]
It follows that $\mathsf{v}(f(r_k)) = \mathsf{v}(F(r_k))$. It remains to prove that $\min\{ \mathsf v(F(a)) \mid a\in R_k\}=\mathsf v(F(r_k))$. So let $b \in R_k$ such that $\mathsf{v}(F(b)) = \min\{ \mathsf v(F(a)) \mid a\in R_k\}$. Then
\[ \mathsf{v}(f(b) - F(b)) = \mathsf{v}(\sum_{i = 0}^{n-1} (b_i - F_i) b^i) > n \geq \mathsf{v}(f(r_k)) = \mathsf{v}(F(r_k)) \geq \mathsf{v}(F(b)),\]
so $\mathsf{v}(F(r_k)) = \mathsf{v}(f(r_k)) \leq \mathsf{v}(f(b)) = \mathsf{v}(F(b))$.
\end{proof}

\section{Combinatorial toolbox}

\begin{notation}\label{notation:interval}
Let $n$ be a positive integer. We write $[n] = \{1,\ldots,n\}$.
\end{notation}

\begin{notation}\label{notation:hyperplanes}
Let $1 < k$, $1 < n_1 \leq \ldots \leq n_k$ be integers, $i,j \in \{1, \ldots,k\}$ with $i < j$, $S \subseteq [n_i]$, and $T \subseteq [n_j]$. We set 
\[ H_{i,j}(S,T) = H_{j,i}(T,S) = [n_1] \times \ldots \times [n_{i-1}] \times S \times [n_{i+1}] \times \ldots \times [n_{j-1}] \times T \times [n_{j+1}] \times \ldots \times [n_k].\]

For $s \in [n_i]$, we define $H_{i,j}(s,T)=H_{i,j}(\{s\},T)$. Moreover, we write $H_{i,j}(S,[n_j]) = H_i(S)$. \\

Note that the $H_i(s)$ are $(k-1)$-dimensional hyperplanes in the grid $[n_1] \times \ldots \times [n_k]$. Analogously, the $H_{i,j}(s,t)$ are $(k-2)$-dimensional subspaces.
\end{notation}

\begin{lemma}\label{lemma:hyperplanes}
Let $k >2$ and $1< n_1 \leq \ldots \leq n_k$ be integers. Let $I \subseteq [n_1] \times \ldots \times [n_k]$. Assume that for every $i \in \{1,\ldots,k\}$ and $r \in [n_i]$ there exists $j \in \{1,\ldots,k\}\setminus \{i\}$ and $T \subseteq [n_j]$ such that $I \cap H_i(r) = H_{i,j}(r,T)$. In other words, every intersection of $I$ with a $(k-1)$-dimensional hyperplane is the union of $(k-2)$-dimensional parallel hyperplanes.

Then there exists $\ell \in \{1, \ldots, k\}$ and $S \subseteq [n_\ell]$ such that $I = H_\ell(S)$. That is, $I$ is the union of $(k-1)$-dimensional parallel hyperplanes.
\end{lemma}

\begin{proof}
If $I = \emptyset$ then the statement is trivial. Assume that $I \neq \emptyset$. Let $s \in [n_1]$ such that $H_1(s) \cap I \neq \emptyset$. By the hypothesis of the lemma there exists $j \in \{2, \ldots,k\}$ and $T_j \subseteq [n_j]$ such that $H_1(s) \cap I = H_{1,j}(s,T_j)$.

\textbf{Case 1.} $T_j = [n_j]$. If we can prove for every $m \in [n_1]$ with $H_1(m) \cap I \neq \emptyset$  that $H_1(m) \subseteq I$, we are done by setting $\ell =1 $ and $S = \{m \in [n_1] \mid H_1(m) \cap I \neq \emptyset \}$. So let $m \in [n_1]$ with $H_1(m) \cap I \neq \emptyset$ and $(m,m_2,\ldots,m_k) \in H_1(m) \cap I$. For $i \in \{2,\ldots,k\}$, we obtain
\begin{align*}
H_{1,i}(s,m_i) \cap I = H_i(m_i) \cap H_1(s) \cap I = H_i(m_i) \cap H_1(s) = H_{1,i}(s,m_i),
\end{align*}
where the second equality follows from $T_j=[n_j]$. Hence, $ H_{i,1}(m_i,s) = H_{1,i}(s,m_i) \subseteq I$ for every $i\in \{ 2, \ldots ,k \}$. 

For fixed $i \in \{2, \ldots,k\}$, by the hypothesis of the lemma, there exist $j_i \in \{1,\ldots,k\} \setminus \{i\}$ and $T_i \subseteq [n_{j_i}]$ such that $I \cap H_i(m_i) = H_{i,j_i}(m_i,T_i)$. Now $H_{i,1}(m_i,s) \subseteq I \cap H_i(m_i) = H_{i,j_i}(m_i,T_i)$, and therefore either $T = [n_i]$ or $j_i =1$, and in both cases $m \in T_i$. So, also in both cases, $I \cap H_i(m_i) = H_{i,j_i}(m_i,T_i) \supseteq H_{i,1}(m_i,m) = H_{1,i}(m,m_i)$ for every $i \in \{2,\ldots,k\}$. 

Again, by the hypothesis of the lemma, there exists $j \in \{2,\ldots,k\}$ and $L \subseteq [n_j]$ such that $H_1(m) \cap I = H_{1,j}(m,L)$. Choose $i \in \{2,\ldots,k\}\setminus \{j\}$. Then $H_{1,i}(m,m_i) \subseteq H_1(m) \cap I = H_{1,j}(m,L)$. It follows that $L = [n_j]$ and hence $H_1(m) \cap I = H_1(m)$.

\textbf{Case 2.} $T_j \neq [n_j]$. Let $i \in \{1,\ldots,k\} \setminus \{1,j\}$. Choose $x \in [n_i]$ arbitrary. Then $H_i(x) \cap I \neq \emptyset$. By the hypothesis of the lemma, there exist $j' \in \{1,\ldots,k\} \setminus \{i\}$ and $L \subseteq [n_{j'}]$ such that $H_i(x) \cap I = H_{i,j'}(x,L)$. Clearly, $H_1(s) \cap H_{i,j'}(x,L) \subseteq I \cap H_1(s) = H_{1,j}(s,T_j)$. Since $j \notin \{1,i\}$ and $T_j \neq [n_j]$, it follows that $j' = j$ and $L \subseteq T_j$ (for otherwise, $H_1(s) \cap H_{i,j'}(x,L) \subseteq  H_{1,j}(s,T_j)$  would contain an element whose $j$-th coordinate is in $[n_j] \setminus T_j$). Hence $H_i(x) \cap I = H_{i,j}(x,T_j)$.

Since $x \in [n_i]$ was chosen arbitrary, we obtain 
\begin{align*}
I = \bigcup_{x \in [n_i]} (H_i(x) \cap I) = \bigcup_{x \in [n_i]} H_{i,j}(x,T_j) = H_j(T_j)
\end{align*}
and we are done choosing $\ell = j$ and $S = T_j$.
\end{proof}

\begin{notation}\label{notation:tensors}
Let $k >1$ and $1< n_1 \leq \ldots \leq n_k$ be integers. By $\mathbb{Q}^{n_1\times \ldots \times n_k}$ we denote the set of all $(n_1 \times \ldots \times n_k)$--arrays, that is, the $k$-dimensional analogues of matrices over $\Q$. Let $M \in \mathbb{Q}^{n_1\times \ldots \times n_k}$. For $i \in [n_1] \times \ldots \times [n_k]$, we write $M_i$ for the entry of $M$ indexed by $i$.

For $I \subseteq [n_1] \times \ldots \times [n_k]$, let $Z_I = \{M \in \mathbb{Q}^{n_1\times \ldots \times n_k} \mid \sum_{i \in I} M_i = 0\}$. Moreover,
\[ Z := \bigcap_{\substack{\ell \in \{1,\ldots k\} \\ r \in [n_\ell]}} Z_{H_\ell(r)}.\]

For instance, if $k = 2$ then $Z$ is the set of all $(n_1 \times n_2)$--matrices with all row and column sums equal to $0$.
\end{notation}

Since the elements of $Z$ are defined by the property that sums over hyperplanes are $0$, clearly, sums over disjoint unions of hyperplanes are also $0$. The next lemma shows that no other sum of a subset of the entries of $M \in Z$ is necessarily $0$. We will use it to show the existence of an array $M \in \mathbb{Q}^{n_1\times \ldots \times n_k}$ such that the sum over a subset of the entries of $M$ is $0$ if and only if the corresponding index set is a disjoint union of hyperplanes.

\begin{lemma}\label{lemma:onesum}
Let $k > 1$ and $1 < n_1 \leq \ldots \leq n_k$ be integers. Let $I \subseteq [n_1] \times \ldots \times [n_k]$ be non-empty such that $I \neq H_\ell(S)$ for all $\ell \in \{1,\ldots,k\}$ and $S \subseteq [n_\ell]$.

Then $Z \setminus Z_I \neq \emptyset$.
\end{lemma}

\begin{proof}
We do an induction on $k$. If $k=2$, we deal with $(n_1 \times n_2)$--matrices and we have to show that there exists a matrix $M \in \Q^{n_1 \times n_2}$ all whose row and column sums are $0$ and such that $\sum_{i \in I} M_i \neq 0$.

Let $I' = I \setminus \bigcup_{H_1(s) \subseteq I} H_1(s)$. Note that $I'$ results from $I$ by removing all rows fully contained in $I$. It suffices to show the assertion for $I'$, so assume without loss of generality that $I = I'$. Now, note that there exists $(i_1,j_1) \in I$ such that neither the $i_1$-th row nor the $j_1$-th column is contained in $I$. Let $i_2 \in [n_1]$ such that $(i_2,j_1) \notin I$. In the same way, let $j_2 \in [n_2]$ such that $(i_1,j_2) \notin I$. Define the matrix $M \in Z$ via
\begin{align*}
M_i = \begin{cases}
1 \text{ if } i \in \{(i_1,j_1), (i_2,j_2)\}, \\
-1 \text{ if } i \in \{(i_1,j_2), (i_2,j_1)\}, \\
0 \text{ otherwise}.
\end{cases}
\end{align*}
Then $\sum_{i \in I} M_i \in \{1,2\}$, hence $M \in Z \setminus Z_I$.

Now, let $k > 2$. Let $i \in \{1, \ldots,k\}$ and $r \in [n_i]$ such that for all $j \in \{1,\ldots,k\}$ and $T \subseteq [n_j]$, the intersection of $I$ and $ H_i(r)$ is not equal to $H_{i,j}(r,T)$. Such $i$ and $r$ exist by Lemma~\ref{lemma:hyperplanes}. Set $J = I \cap H_i(r)$. By the induction hypothesis, we find a $(k-1)$-dimensional array $N$ indexed by elements of $H_i(r)$ such that $\sum_{d \in H_{i,j}(r,s)} N_d = 0$ for all $j \in \{1,\ldots,k\}\setminus \{i\}$ and $s \in [n_j]$, and $\sum_{d \in J} N_d \neq 0$. We define
\begin{align*}
M_d = \begin{cases}
N_d \text{ if } d \in H_i(r), \\
0 \text{ otherwise}.
\end{cases}
\end{align*}
Then $M \in Z \setminus Z_I$.
\end{proof}

We use the following fact from linear algebra.

\begin{fact}\label{fact:subspaces}
Let $K$ be an infinite field and $V$ a $K$-vector space. Let $W,W_1,\ldots,W_n$ be linear subspaces of $V$ such that $W \subseteq \bigcup_{j=1}^n W_j$.

Then there exists $j \in \{1,\ldots,n\}$ such that $W \subseteq W_j$.
\end{fact}

\begin{proof}
This is a known fact from linear algebra~\cite[Theorem 1.2]{roman}.
\end{proof}

\begin{proposition}\label{proposition:existence}
Let $k$ be a positive integer and $1 < n_1 \leq \ldots \leq n_k$ integers. Then there exists $M \in Z$ such that $M \in Z_I$ only if $I$ is a disjoint union of hyperplanes, that is, only if $I = H_\ell(S)$ for some $\ell \in \{1,\ldots,k\}$ and $S \subseteq [n_\ell]$.
\end{proposition}

\begin{proof}
Note that $Z_I$ is a subspace of $\Q^{n_1 \times \ldots \times n_k}$ for every $I \subseteq [n_1] \times \ldots \times [n_k]$. Assume that the statement of the proposition does not hold. Then there are finitely many non-empty $I_1,\ldots,I_n \subseteq [n_1] \times \ldots \times [n_k]$ each of which is not a union of parallel hyperplanes such that $Z \subseteq \bigcup_{j = 1}^n Z_{I_j}$. It follows by Fact~\ref{fact:subspaces} that $Z \subseteq Z_{I_j}$ for some $j$. This is a contradiction to Lemma~\ref{lemma:onesum}.
\end{proof}

\begin{notation}\label{notation:pairsoftensors}
Let $k >1$ and $1< n_1 \leq \ldots \leq n_k$ be integers. For $I, J \subseteq [n_1] \times \ldots \times [n_k]$, let $Z_{I,J} = \{(M,N) \in (\mathbb{Q}^{n_1\times \ldots \times n_k})^2 \mid \sum_{i \in I} M_i = \sum_{j\in J} N_j\}$. Moreover,
\[ \overline Z := \bigcap_{\substack{\ell \in \{1,\ldots k\} \\ r \in [n_\ell]}} Z_{H_\ell(r), H_\ell(r)}.\]
For $s\in \Q$ we denote by $(s)$ the array all of whose entries are $s$.
\end{notation}

\begin{lemma}\label{lemma:avoidance}
Let $k > 1$ and $1 < n_1 \leq \ldots \leq n_k$ be integers. Let $I,J \subseteq [n_1] \times \ldots \times [n_k]$ be non-empty, and not both the same union of parallel hyperplanes.

Then $\overline Z \setminus Z_{I,J} \neq \emptyset$.
\end{lemma}
\begin{proof}
First, suppose $I\neq J$ and assume, without loss of generality, that there exists $i\in I\setminus J$. Now let $M$ be such that $M_j=\delta_{i,j}$, where $\delta_{i,j}$ is the Kronecker-$\delta$. Then $(M,M)\in \overline Z\setminus Z_{I,J}$.

Now suppose that $I=J\neq H_\ell(S)$ for all $\ell \in \{1,\ldots,k\}$ and $S \subseteq [n_\ell]$. Then $(M,(0))\in \overline Z\setminus Z_{I,J}$, where $M$ is an array as in Proposition \ref{proposition:existence}.
\end{proof}

\begin{proposition}\label{proposition:existenceofpairs}
Let $k$ be a positive integer and $1 < n_1 \leq \ldots \leq n_k$ integers. Then there exists $(M,N) \in \overline Z$ such that $(M,N) \in Z_{I,J}$ only if $I=J$ is a disjoint union of parallel hyperplanes (that is, only if $I = H_\ell(S)$ for some $\ell \in \{1,\ldots,k\}$ and $S \subseteq [n_\ell]$). Moreover, $M$ and $N$ can be chosen such that all of their entries are positive integers.
\end{proposition}
\begin{proof}
Note that $Z_{I,J}$ is a subspace of $(\Q^{n_1 \times \ldots \times n_k})^2$ for all $I,J \subseteq [n_1] \times \ldots \times [n_k]$. Assume that the statement of the proposition does not hold. Then there are finitely many pairs $(I_1, J_1),\ldots,(I_n,J_n)$ of non-empty subsets of $[n_1] \times \ldots \times [n_k]$  such that in each case $I_j$ and $J_j$ are not the same union of parallel hyperplanes, and such that $\overline Z \subseteq \bigcup_{j = 1}^n Z_{I_j, J_j}$. It follows by Fact~\ref{fact:subspaces} that $\overline Z \subseteq Z_{I_j, J_j}$ for some $j$. This is a contradiction to Lemma~\ref{lemma:avoidance}.

Moreover, by multiplying with the common denominator of all the entries from $M$ and $N$, we obtain integer entries in both arrays. Now let $\eta$ be a positive integer with
\[\eta>\max\{|\sum_{i \in I} M_i - \sum_{j\in J} N_j|: I,J\subseteq [n_1] \times \ldots \times [n_k]\}.\]

Then we obtain that each entry of $M'=M+(\eta)$ and $N'=N+(\eta)$ is a positive integer. It remains to show that $(M',N')$ has the asserted property. We show for any $I,J$ that $(M,N)\in Z_{I,J}$ if and only if $(M',N')\in Z_{I,J}$. We already know that $(M,N)\in Z_{I,J}$ implies that $I = J$. So the equalities \[\sum_{i \in I} M'_i =(\sum_{i \in I} M_i) + |I|\eta = (\sum_{j\in J} N_j) + |J|\eta= \sum_{j\in J} N'_j \] imply $(M',N')\in Z_{I,J}$. To see the converse, let $I,J$ be such that $(M',N')\in Z_{I,J}$. Since
\[
(\sum_{i \in I} M_i) + |I|\eta =\sum_{i \in I} M'_i = \sum_{j\in J} N'_j = (\sum_{j\in J} N_j) + |J|\eta,
\]
we obtain that
\[
\sum_{i \in I} M_i - \sum_{j\in J} N_j = (|J|-|I|)\eta.
\]
By the choice of $\eta$ and the pair $(M,N)$ this is only possible if $I=J=H_\ell(S)$ for some $\ell \in \{1,\ldots,k\}$ and $S \subseteq [n_\ell]$.
\end{proof}

\section{Sets of lengths of integer-valued polynomials over a DVR}

We use a special case of a lemma from a paper by Nakato, Rissner and the second author \cite[Lemma 3.4]{globalcase}.

\begin{lemma}\label{lemma:old}(\cite[Lemma 3.4]{globalcase} specialized to DVRs)
Let $V$ be a discrete valuation domain with quotient field $K$, normalized valuation $\mathsf v: K^\times \to \Z$, prime element $\pi \in V$ and let $H \in \Int(V)$ be of the following form:
\[ H = \frac{\prod_{i \in I} f_i}{\pi^e} \text{ with } \min\{\mathsf{v}(\prod_{i \in I} f_i(a)) \mid a \in V\} = e, \]
where $e$ is a positive integer and, for each $i \in I$, $f_i \in V[X]$ is irreducible in $K[X]$.

If $H = g_1 \cdots g_\nu$ is a factorization of $H$ into (not necessarily irreducible) non-units of $\Int(V)$ then each $g_j$ is of the form 
\[ g_j = \frac{\prod_{i \in I_j} f_i}{\pi^{e_j}}, \]
where $\emptyset \neq I_j \subseteq I$ and the $e_j$ are non-negative integers such that $I$ is the disjoint union of the $I_j$ and $\sum_{i = 1}^\nu e_i = e$.
\end{lemma}

\begin{lemma}\label{lemma:irreducibles}
Let $V$ be a discrete valuation domain with finite residue field of cardinality $q$. Let $K$ be the quotient field of $V$, $\mathsf v: K^\times \to \Z$ its normalized valuation and $\pi \in V$ a prime element, that is, $\mathsf v(\pi) =1$. By $R_1,\ldots,R_q$ denote the residue classes of $V$ modulo the maximal ideal.

Let $F_1,\ldots,F_r,G_1,\ldots,G_s \in V[X]$ irreducible over $K$ such that
\begin{itemize}
\item[(i)] $e := \min\{ \sum_{i = 1}^r \mathsf v(F_i(a)) \mid a \in R_1\} = \min \{ \sum_{j = 1}^s \mathsf v(G_j(a))\mid a \in R_m\}$ for all $m \in \{2,\ldots,q\}$,


\item[(ii)] $\min \{ \mathsf v(G_j(a)) \mid a \in R_1\} = 0$ for all $j \in \{1,\ldots,s\}$,
\item[(iii)] $\min \{ \mathsf v(F_i(a)) \mid a \in R_m\} = 0$ for all $i \in \{1,\ldots,r\}$ and $m \in \{2,\ldots,q\}$.

\item[(iv)] For $m \in \{1,\ldots,q\}$, there exists $r_m \in R_m$ such that $\min\{ \mathsf v(F_i(a)) \mid a \in R_m\} = \mathsf{v}(F_i(r_m))$ and $\min \{ \mathsf v(G_j(a))\mid a \in R_m\} = \mathsf{v}(G_j(r_m))$ for all $i$ and $j$.
\end{itemize}
Define
\[ H = \frac{( \prod_{i =1}^r F_i ) ( \prod_{j =1}^s G_j )}{\pi^e}.\] 

Then $H \in \Int(V)$.
Furthermore, $H$ is a product of two non-units in $\Int(V)$ if and only if there exist non-empty $I \subsetneqq \{1,\ldots,r\}$ and $J \subsetneqq \{1,\ldots,s\}$ such that  $\min \{ \sum_{i \in I} \mathsf v(F_i(a)) \mid a \in R_1\} = \min \{ \sum_{j \in J} \mathsf v(G_j(a)) \mid a \in R_m\}$ for all $m \in \{2,\ldots,q\}$.
\end{lemma}

\begin{proof}
Clearly $H$ is integer-valued over $V$. If there exist $I$ and $J$ as in the lemma, define $e' = \min_{a \in R_1} \sum_{i \in I} \mathsf v(F_i(a)) = \sum_{i \in I} \mathsf v(F_i(r_1)) = \min_{a \in R_m} \sum_{j \in J} \mathsf v(G_j(a)) = \sum_{j \in J} \mathsf v(G_j(r_m))$ (for $m\in \{2,\ldots, q\}$). Then clearly 
\begin{align*}
H = \frac{( \prod_{i \in I} F_i ) ( \prod_{j \in J} G_j )}{\pi^{e'}} \cdot \frac{( \prod_{i \in \{ 1, \ldots, r \}\setminus I} F_i ) ( \prod_{j \in \{ 1, \ldots,s \} \setminus J} G_j )}{\pi^{e-e'}}
\end{align*}
is a decomposition into two non-units in $\Int(V)$.

Conversely, if $H = H_1 \cdot H_2$ is a decomposition of $H$ where $H_1$ and $H_2$ are non-units of $\Int(V)$, then, by Lemma~\ref{lemma:old}, there exist non-empty $I \subsetneqq \{1,\ldots,r\}$ and $J \subsetneqq \{1,\ldots,s\}$ such that 
\[ H_1 = \frac{( \prod_{i \in I} F_i ) ( \prod_{j \in J} G_j )}{\pi^{e'}} \text{ and } H_2 = \frac{( \prod_{i \in \{ 1, \ldots, r \}\setminus I} F_i ) ( \prod_{j \in \{ 1, \ldots,s \} \setminus J} G_j )}{\pi^{e-e'}} \]
for some $e' \in \{0,\ldots,e\}$. Assume to the contrary that $\min_{a \in R_1} \sum_{i \in I} \mathsf v(F_i(a)) \neq \min_{a \in R_m} \sum_{j \in J} \mathsf v(G_j(a))$ for some $m \in \{2,\ldots,q\}$. Exchanging, if necessary, the roles of $I$ and $\{ 1, \ldots, r \}\setminus I$ respectively $J$ and $\{ 1, \ldots,s \} \setminus J$, we may assume without loss of generality that $\min_{a \in R_1} \sum_{i \in I} \mathsf v(F_i(a)) > \min_{a \in R_m} \sum_{j \in J} \mathsf v(G_j(a))$. Since $H_1$ is an integer-valued polynomial on $V$, it follows that 
\[ \min_{a \in R_1} \sum_{i \in I} \mathsf v(F_i(a)) > \min_{a \in R_m} \sum_{j \in J} \mathsf v(G_j(a)) \geq e'.\] 
Hence we get 
\[ \min_{a \in R_1} \sum_{i \in \{ 1, \ldots, r \}\setminus I} \mathsf v(F_i(a)) < \min_{a \in R_m} \sum_{j \in \{ 1, \ldots,s \} \setminus J} \mathsf v(G_j(a)) \leq e-e', \]
which is a contradiction because $H_2 \in \Int(V)$. 
\end{proof}

By a \textit{pair of ordered partitions} of sets $I$ and $J$, we mean an equivalence class of tuples $((I_1,\ldots,I_\ell),(J_1,\ldots,J_\ell))$, where the $I_\lambda$ form a partition of $I$ and the $J_\lambda$ of $J$, under the equivalence relation where $((I_1,\ldots,I_\ell),(J_1,\ldots,J_\ell))$ is identified with $((I_{\sigma(1)},\ldots,I_{\sigma(\ell)}),(J_{\sigma(1)},\ldots,J_{\sigma(\ell)}))$ for all permutations $\sigma$ of $\{1,\ldots,\ell\}$. The positive integer $\ell$ is called the \textit{length} of the pair of ordered partitions.

\begin{lemma}\label{lemma:factorizations}
Let $V$ be a discrete valuation domain with finite residue field of cardinality $q$ and residue classes $R_1,\ldots,R_q$. Let $K$ be the quotient field of $V$, $\mathsf v: K^\times \to \Z$ its normalized valuation and $\pi \in V$ a prime element, that is, $\mathsf v(\pi) =1$. Let $F_1,\ldots,F_r,G_1,\ldots,G_s \in V[X]$ irreducible over $K$ and pairwise non-associated over $K$ such that
\begin{itemize}
\item[(i)] $e := \min\{ \sum_{i = 1}^r \mathsf v(F_i(a)) \mid a \in R_1\} = \min \{ \sum_{j = 1}^s \mathsf v(G_j(a))\mid a \in R_m\}$ for all $m \in \{2,\ldots,q\}$,


\item[(ii)] $\min \{ \mathsf v(G_j(a)) \mid a \in R_1\} = 0$ for all $j \in \{1,\ldots,s\}$,
\item[(iii)] $\min \{ \mathsf v(F_i(a)) \mid a \in R_m\} = 0$ for all $i \in \{1,\ldots,r\}$ and $m \in \{2,\ldots,q\}$.
\item[(iv)] For $m \in \{1,\ldots,q\}$, there exists $r_m \in R_m$ such that $\min\{ \mathsf v(F_i(a)) \mid a \in R_m\} = \mathsf{v}(F_i(r_m))$ and $\min \{ \mathsf v(G_j(a))\mid a \in R_m\} = \mathsf{v}(G_j(r_m))$ for all $i$ and $j$.
\end{itemize}
Define
\[ H = \frac{( \prod_{i =1}^r F_i ) ( \prod_{j =1}^s G_j )}{\pi^e}.\] 
Then, for every $\ell >0$, a bijective correspondence between, on the one hand, pairs of ordered partitions $ (I_1,\ldots,I_\ell)$ of $\{1,\ldots,r\}$ and $(J_1, \ldots, J_\ell)$ of $\{1,\ldots,s\}$ satisfying
\begin{itemize}
\item[(1)] $e_\lambda:=\min \{ \sum_{i \in I_\lambda} \mathsf v(F_i(a)) \mid a \in R_1\} = \min \{ \sum_{j \in J_\lambda} \mathsf v(G_j(a)) \mid a \in R_m\}$ for all $m \in \{2,\ldots,q\}$ and $\lambda \in \{1,\ldots,\ell\}$, and

\item[(2)] for all $\lambda \in \{1,\ldots,\ell\}$ and all non-empty $I \subsetneqq I_\lambda$ and $J\subsetneqq J_\lambda$ there is $m \in \{2,\ldots,q\}$ such that $\min \{ \sum_{i \in I} \mathsf v(F_i(a)) \mid a \in R_1\} \neq \min \{ \sum_{j \in J} \mathsf v(G_j(a)) \mid a \in R_m\}$

\end{itemize}
and, on the other hand, essentially different factorizations of $H$ into $\ell$ irreducible elements of $\Int(V)$, is given by
\begin{align*}
((I_1, \ldots, I_\ell),  (J_1, \ldots, J_\ell)) \mapsto  \frac{( \prod_{i \in I_1} F_i ) ( \prod_{j \in J_1} G_j )}{\pi^{e_1}} \cdot \ldots \cdot \frac{( \prod_{i \in I_\ell} F_i ) ( \prod_{j \in J_\ell} G_j )}{\pi^{e_\ell}}.
\end{align*} 
\end{lemma}

\begin{proof}
This follows immediately from Lemma~\ref{lemma:irreducibles}.
\end{proof}


\begin{theorem}\label{theorem:setsoflenghts}
Let $V$ be a discrete valuation domain with finite residue field. Suppose that the quotient field $K$ of $V$ admits a valuation ring independent from $V$ whose maximal ideal is principal or whose residue field is finite. Let $k$ be a positive integer and $1 < n_1 \leq \ldots \leq n_k$ integers. 

Then there exists an integer-valued polynomial $H \in \Int(V)$ which has precisely $k$ essentially different factorizations into irreducible elements of $\Int(V)$ whose lengths are exactly $n_1,\ldots,n_k$.
\end{theorem}

\begin{proof}
If $k=1$ then $H= X^{n_1}$ has the desired property. So let $k\geq 2$. By Lemma~\ref{lemma:factorizations} it suffices to construct polynomials $F_i$ and $G_j$ as in the hypothesis of this lemma such that there are exactly $k$ different pairs of ordered partitions of the index sets satisfying the conditions of the lemma; and their lengths are exactly $n_1,\ldots,n_k$.

We set $r = s = n_1 \cdots n_k$. Let $(M,N) \in (\N^{n_1\times \ldots \times n_k})^2$ be a pair of arrays such that $(M,N) \in \overline Z$ and $(M,N) \notin Z_{I,J}$ for non-empty $I,J \subseteq [n_1]\times \ldots \times [n_k]$ unless $I$ and $J$ are the same union of parallel hyperplanes (see Notation~\ref{notation:hyperplanes}, \ref{notation:tensors} and \ref{notation:pairsoftensors}). Such a pair exists by Proposition~\ref{proposition:existenceofpairs}.  Recall that $M_i$ for $i \in [n_1]\times \ldots \times [n_k]$ denotes the $i$-th entry of $M$.


Let $R_1,\ldots,R_q$ be the residue classes of $V$ and $\mathsf{v}: K^\times \to \Z$ its normalized valuation. Let $r_m \in R_m$ be arbitrary for each $m \in \{1,\ldots,q\}$. Since the $R_m$ are infinite, we can pick, for each $i \in [n_1]\times \ldots \times [n_k]$, a set of $M_i$ distinct elements $a_1^{i,1},\ldots,a_{M_i}^{i,1}\in V$ with $\mathsf{v}(r_1 - a_j^{i,1}) = 1$ for all $j$, and for each $m \in \{2,\ldots,q\}$ a set of $N_i$ distinct elements $a_1^{i,m},\ldots,a_{N_i}^{i,m}\in V$ with $\mathsf{v}(r_m - a_j^{i,m}) = 1$ for all $j$, and such that $a_{j_1}^{i_1,m_1} = a_{j_2}^{i_2,m_2}$ implies $j_1 = j_2$, $i_1 = i_2$ and $m_1 = m_2$.

For each $i \in [n_1]\times \ldots \times [n_k]$, we set 
\begin{align*}
f_i &= \prod_{j = 1}^{M_i} (X - a_j^{i,1}), \\
g_i &= \prod_{m = 2}^q \prod_{j=1}^{N_i} (X-a_j^{i,m}).
\end{align*}
By Lemma~\ref{lemma:glueing} and Lemma~\ref{lemma:glueing-finite}, there exist, for each $i \in [n_1]\times \ldots \times [n_k]$, polynomials $F_i, G_i \in V[X]$, irreducible in $K[X]$, which, furthermore, by construction are all pairwise non-associated in $K[X]$ and for all non-empty $I,J \subseteq  [n_1]\times \ldots \times [n_k]$, satisfy the system of equalities
\begin{equation}\label{equation}
\min_{a \in R_1} \sum_{i \in I} \mathsf v(F_i(a)) = \min_{a \in R_m} \sum_{j \in J} \mathsf v(G_j(a)) \qquad (m = 2, \ldots, q)
\end{equation}
if and only if $I$ and $J$ are both the same union of parallel hyperplanes. Indeed, for all $i \in [n_1]\times \ldots \times [n_k]$ and $m \in \{1,\ldots,q\}$, we have $M_i = \min_{a \in R_m} \mathsf{v}(f_i(a)) = \mathsf{v}(f_i(r_m)) = \mathsf{v}(F_i(r_m)) =\min_{a \in R_m} \mathsf{v}(F_i(a))$ and analogously for $N_i$, $g_i$ and $G_i$.

Hence the equation in (\ref{equation}) holds if and only if $\sum_{i \in I} M_i = \sum_{j \in J} N_j$ which is the case if and only if $I = J$ is a union of parallel hyperplanes.
So, the only admissible (in the sense of Lemma~\ref{lemma:factorizations}) pairs of ordered partitions of the index set $[n_1]\times \ldots \times [n_k]$, namely, those that correspond to factorizations into irreducibles, are the ones of the form $((H_r(1), \ldots, H_r(n_r)),(H_r(1), \ldots, H_r(n_r)))$ for $r \in \{1,\ldots,k\}$. These are exactly $k$ many of lengths $n_1,\ldots,n_k$.
\end{proof}

\begin{corollary}
The conclusion of Theorem~\ref{theorem:setsoflenghts} holds in each of the following cases:
\begin{itemize}
\item[(1)] $V$ is a valuation ring of a global field.
\item[(2)] $V$ is a discrete valuation domain with finite residue field such that the quotient field of $V$ is a purely transcendental extension of an arbitrary field.
\item[(3)] $V$ is a discrete valuation domain with finite residue field such that the quotient field $K$ of $V$ is a finite extension of a field $L$ that admits a valuation ring independent from $V \cap L$ whose maximal ideal is principal or whose residue field is finite.
\end{itemize}

That is, in each of these three cases, for all positive integers $k$ and $1 < n_1 \leq \ldots \leq n_k$, there exists an integer-valued polynomial $H \in \Int(V)$ which has precisely $k$ essentially different factorizations into irreducible elements of $\Int(V)$ whose lengths are exactly $n_1,\ldots,n_k$.

\end{corollary}

\begin{proof}
(1) Note that a global field has infinitely many non-equivalent discrete valuations and each valuation ring is discrete. 

(2) The quotient field $K$ of $V$ is also the quotient field of a polynomial ring in one variable over some field and, therefore, $K$ admits infinitely many discrete valuations. 

(3) Let $W_L$ be a valuation domain of $L$ independent from $V \cap L$. Let $W$ be a valuation domain of $K$ extending $W_L$. Then $W$ and $V$ are independent. When the maximal ideal of $W_L$ is principal or its residue field is finite, the same property follows for $W$ by general facts on extensions of valuations~\cite[Chapter VI, § 8.3, Theorem 1]{Bourbaki} and Remark~\ref{remark:valuation}.
\end{proof}

Unfortunately, our construction (including Lemma~\ref{lemma:glueing} and Lemma~\ref{lemma:glueing-finite}) fails for Henselian valued fields, so, in particular, for local fields.
It is therefore natural to pose the following

\begin{problem}
Determine multisets of lengths of factorizations of elements $f \in \Int(V)$, where $V$ is the discrete valuation ring of a Henselian valued field.
\end{problem}

\bibliographystyle{amsplainurl}
\bibliography{bibliography}

\vspace{0.5cm}
\noindent
\textsc{Victor Fadinger, Institute for Mathematics and Scientific Computing, Universität Graz, Heinrichstraße 36, 8010 Graz, Austria} \\
\textit{E-mail address}: \texttt{victor.fadinger@uni-graz.at}\\

\noindent
\textsc{Sophie Frisch, Department of Analysis and Number Theory (5010), Technische Universität Graz, Kopernikusgasse 24, 8010 Graz, Austria} \\
\textit{E-mail address}: \texttt{frisch@math.tugraz.at} \\

\noindent
\textsc{Daniel Windisch, Department of Analysis and Number Theory (5010), Technische Universität Graz, Kopernikusgasse 24, 8010 Graz, Austria} \\
\textit{E-mail address}: \texttt{dwindisch@math.tugraz.at}

\end{document}